\documentclass{amsart}
\usepackage{amsmath,amssymb,hyperref,mathrsfs,graphicx}
\usepackage[utf8x]{inputenc}
\usepackage{amsmath}
\usepackage{amsfonts}
\usepackage{latexsym}
\usepackage{amsthm}
\usepackage{amssymb,amscd}
\usepackage{xargs}
\usepackage{tikz}
\usepackage{stmaryrd}
\usepackage{pgf,tikz}
\usepackage{fancybox}
\usepackage{graphicx}
\usepackage{stmaryrd}
\usepackage{color}
\usepackage{enumitem}

\usepackage[colorinlistoftodos]{todonotes}
 \presetkeys{todonotes}%
{inline,backgroundcolor=gray!20,bordercolor=gray!30}{}
\tikzset{/tikz/notestyleraw/.append style={text=black}}

\newtheorem{thm}{Theorem}[section]

\newtheorem{lem}[thm]{Lemma}
\newtheorem{defn}[thm]{Definition}
\newtheorem{prop}[thm]{Proposition}
\newtheorem{cor}[thm]{Corollary}

\newtheorem{ex}[thm]{Example}
\newtheorem{rmk}[thm]{Remark}
\newcommand{\be}{\begin{eqnarray}}
\newcommand{\ee}{\end{eqnarray}}
\newcommand{\ben}{\begin{eqnarray*}}
\newcommand{\een}{\end{eqnarray*}}
\newcommand{\beal}{\begin{aligned}}
\newcommand{\enal}{\end{aligned}}
\newcommand{\beq}{\begin{equation}}
\newcommand{\eeq}{\end{equation}}

\newcommand{\lb}{\lambda}

\newcommand{\R}{\mathbb{R}}

\newcommand{\N}{\mathbb{N}}

\newcommand{\Z}{\mathbb{Z}}
\newcommand{\Lb}{\Lambda}
\newcommand{\om}{\omega}
\newcommand{\Om}{\Omega}
\newcommand{\dt}{\delta}

\newcommand{\cC}{\mathcal{C}}
\newcommand{\cP}{\mathcal{P}}

\newcommand{\cM}{\mathcal{M}}

\newcommand{\cL}{\mathcal{L}}

\newcommand{\wt}{\widetilde }

\title{Generalized comparison principle for contact Hamilton-Jacobi equations}
\thanks{$\dagger$ {\it Statements and Declarations: }The authors declare no competing interests.}
\subjclass[2010]{35B40, 35B51, 35F21, 37J50, 37J55, 49L25}
\keywords{viscosity solution, contact Hamilton-Jacobi equation,  Mather measure, weak KAM theory}
\begin{document}
\maketitle

\centerline{Gengyu Liu$^*$,\quad Jianlu Zhang$\dagger$}
\medskip
{\footnotesize
\centerline{State Key Laborotary of Mathematical Sciences}
 \centerline{Academy of Mathematics and Systems Science}
 \centerline{Chinese Academy of Sciences, Beijing 100190, China}
  \centerline{{\it Email: }liugengyu@amss.ac.cn$^*$,\quad  jellychung1987@gmail.com$\dagger$}  
}


\begin{abstract}
In this paper, we discuss all the possible pairs $(u,c)\in C(M,\R)\times\R$ solving (in the sense of viscosity) the contact Hamilton-Jacobi equation 
\[
H (x, d_xu, u) = c,\quad x\in M
\]
of which $M$ is a closed manifold and the continuous Hamiltonian $H: (x,p,u)\in T^*M\times\R\rightarrow\R$ is convex, coercive in $p$ but merely non-decreasing in $u$. Firstly, we propose a comparison principle for solutions by using the dynamical information of Mather measures. We then describe the structure of $\mathfrak C$ containing all the $c\in\R$ makes previous equation solvable. We also propose examples to verify the optimality of our approach.  
\end{abstract}

\section{Introduction}\label{s1}

Suppose $M$ is a $n-$dimensional closed manifold equipped with a Riemannian metric $|\cdot|_{x\in M}$. For any continuous {\it contact Hamiltonian} $H:(x,p,u)\in T^* M\times\R\rightarrow\R$
satisfying the following assumptions:\medskip
 \begin{description}
	\item[H1]  $H$ is convex in $p\in T_x^*M$ for any $(x,u)\in M\times\R$;
	\item[H2] $H$ is superlinear in each fiber, i.e. $\lim_{|p|\rightarrow+\infty}{H (  x , p , u )}/{|p|_x} =+\infty$ for any $(x,u)\in M\times\R$;
	\item[H3]  $H(x,p,u)$ is non-decreasing in $u$ for any $(x,p)\in T^*M$; 
			\end{description}
The following {\it  contact Hamilton-Jacobi (H-J) equation} 
\beq\label{eq:hjc}
H(x, d_x u,u)=c,\quad x\in M
\eeq
will be studied. We try to seek all the possible pairs $(u,c)\in C(M,\R)\times\R$ solving \eqref{eq:hjc} in the sense of viscosity. Precisely,  equation \eqref{eq:hjc} is solvable only for $c\in\mathfrak C$, where the {\it admissible set} $\mathfrak C\subset\R$ is a non-empty, connected interval (see Proposition \ref{prop:mane} for the proof). Another feature of \eqref{eq:hjc} is that, for a fixed $c\in \R$ the equation \eqref{eq:hjc} may have different solutions (even by adding additive constant). It is therefore meaningful to understand why the multiplicity happens and how it relates with the structure of $\mathfrak C$.

In this paper, we propose a {\it comparison principle} to explain the multiplicity of solutions. 
We will show how certain {\it Mather measures} dominate the order of solutions and the structure of $\mathfrak C$. Recall that for Hamiltonians independent of $u$, the significance of Mather measures in dominating the solutions have been discussed in \cite{F,FS}, as a part of the weak KAM theory. Besides, for Hamiltonians strictly increasing in $u$, the solution to \eqref{eq:hjc} is unique, which was proved earlier, see \cite{CL,CEL,CIL} for instance. These works prompt us to study the essential role of 
{\bf H3} in ordering the solutions of \eqref{eq:hjc}. Besides, examples are constructed to verify the optimality of our assumptions.

Finally, we point out that a similar comparison principle was made in an earlier work \cite{JMT}, where the authors used a {\it nonlinear adjoint method} developed in \cite{MT1,MT2}. In some sense, their work inspires the comparison principle part of this paper, although our treatment uses a different technique (from the weak KAM theory) which relies on weaker assumptions. At the same time, a description of $\mathfrak C$ was also made in \cite{WY} by using a min-max formula, which doesn't involve the information of solutions. \medskip

Throughout the paper, solution (resp. subsolution, supersolution) is always meant by the viscosity solution (resp. viscosity subsolution, viscosity supersolution). 

\subsection{Main Results}

For any $\theta\in\R$ and Hamiltonian $H$ satisfying {\bf H1-H3}, there exists a unique  {\it ergodic constant} $c(\theta)\in\R$ defined by 
\beq\label{eq:e-const}
c(\theta):=\inf\{c\in\R: H(x,d_x\om,\theta)\leq c \text{ solvable in the sense of viscosity on $M$}\},
\eeq
see e.g. \cite{Tr}. We can state a standard result in light of the classical Perron method:
\begin{prop}[ergodic constant]\label{prop:mane}
Assume {\bf H1-H3}, then 
\begin{itemize}
\item[1)] $c:\R\rightarrow\R$ is continuous and non-decreasing;
\item[2)] Denote by $\mathfrak C:=\{c(a)|a\in\R\}$, then for any $c\in\mathfrak C$, the associated equation \eqref{eq:hjc} admits a solution. 
\end{itemize}
\end{prop}
For any $c\in\mathfrak C$
there always exists a single interval $ I(c)\subset\R$ such that $c(\theta)=c$ for any $\theta\in I(c)$. To state our comparison principle, we also need the following assumptions:

 \begin{description}
	\item[H4]  $\partial_uH(x,p,u)$ exists and is continuous for any $(x,p,u)\in T^*M\times\R$; 
	\item[H4'] $\partial_uH(x,p,\theta)$ exists and is continuous for some $\theta\in I(c)$;
	\item[H5]  $H(x,p,u)$ is convex in $u$ for any $(x,p)\in T^*M$.
			\end{description}

\begin{thm}[Local]\label{thm:l-cp}
Assume {\bf H1-H4}. 
\begin{itemize}
\item[1)]If $u_1,u_2$ are two solutions of \eqref{eq:hjc} which satisfy 
\beq\label{eq:cp-mea-2}
\int_{TM} u_1(x) d\mu(x,v)\leq \int_{TM} u_2(x) d\mu(x,v),\quad \forall\; \mu\in\mathfrak M_-(u_1),
\eeq
then $u_1\leq u_2$. The set $\mathfrak M_-(u_1)$ which contains {\bf the ordinal Mather measures} associated with $u_1$ is defined in \eqref{eq:u-tint} of Definition \ref{defn:t-mea}.
\item[2)] If $\mathfrak M_-(u)=\emptyset$ for some  solution $u$ of \eqref{eq:hjc}, then $u$ has to be the unique solution of \eqref{eq:hjc}. Consequently, there exists at most one solution $u$ of \eqref{eq:hjc} such that $\mathfrak M_-(u)=\emptyset$.
\end{itemize}
\end{thm}
\begin{thm}[Global]\label{thm:g-cp}
Assume {\bf H1-H3, H4', H5}. 
\begin{itemize}
\item[1)] If $u_1,u_2$ are two solutions of \eqref{eq:hjc} which satisfy 
\beq\label{eq:cp-mea}
\int_{TM} u_1(x) d\mu(x,v)\leq \int_{TM} u_2(x) d\mu(x,v),\quad \forall\; \mu\in\mathfrak M_-^\theta,
\eeq
then $u_1\leq u_2$. The set $\mathfrak M_-^\theta$ which contains {\bf the ordinal Mather measures} associated with $\theta\in I(c)$ is defined in \eqref{eq:c-tint} of Definition \ref{defn:t-mea}.
\item[2)] If $\mathfrak M_-^\theta=\emptyset$, there exists a unique solution of \eqref{eq:hjc}. \item[3)] If $\mathfrak M_-^\theta=\{\mu\}$ is a singleton,  then any two solutions are comparable. In this case, if two solutions $u_1, u_2$ satisfy  
\[
\int_{TM} u_1(x) d\mu(x,v)=\int_{TM} u_2(x) d\mu(x,v), 
\]
then $u_1=u_2$.
\end{itemize}
\end{thm}




\begin{rmk}
\begin{itemize}
\item[(i)] 
A different version of Theorem \ref{thm:l-cp} was proved in  \cite{JMT} under stronger  assumptions. Here we make the results optimal in requiring weaker assumptions and our definition of $\mathfrak M_-(u)$ is quite different. Later we can see that our definition is invariant w.r.t. certain contact Hamilton's equation in 
Theorem \ref{thm:mat-set}.  As far as we know, Theorem \ref{thm:g-cp} is new in the literature. 

\item[(ii)] Assumption {\bf H2} can be weakened to a coercive condition:\vspace{3pt}

\noindent{\bf H2':} For any $(x,u)\in M\times\R$, there holds
	$\lim_{|p|\rightarrow+\infty}H (  x , p , u ) =+\infty$.\vspace{3pt}
	
\noindent Then aforementioned conclusions still hold. This is because once  $c\in\mathfrak C$ and solutions $u_1, u_2$ of \eqref{eq:hjc} established, we can always modify the Hamiltonian $H$ for $|p|\gg 1$, such that the modified Hamiltonian is superlinear (or even quadratic) in $p$, which still possesses these two solutions (see \cite{CCIZ,DFIZ} for a standard argument). 
\item[(iii)] In Section \ref{s3}, we propose several examples to verify the necessity of our assumptions. In our context, conditions \eqref{eq:cp-mea-2} and \eqref{eq:cp-mea} seem to be essential and optimal for the comparison principle. 
\end{itemize}
\end{rmk}

\subsection{Contact Hamiltonian dynamics of the Mather set}
For $C^2-$smooth contact Hamiltonians, a unique {\it contact vector field} $X_H: T^*M\times\R\rightarrow T(T^*M\times\R)$ can be established by 
\[
\cL_{X_H}\alpha=-\partial_u H(x,p,u) \alpha, \quad -H(x,p,u)=\alpha(X_H)
\]
where $\cL_{X_H}$ is the Lie derivative operator with respect to $X_H$, $\alpha:= du-p dx$ is called the {\it contact $1-$form} and $-\partial_u H(x,p,u)$ is called the {\it multiplier} of $H$ \cite{FQ}. In $(x,p,u)-$coordinates, we can formulate the {\it contact Hamilton's equation} (associated with $X_H$) by
\beq\label{eq:ode}
\left\{
\beal
&\dot x=\partial_p H(x,p,u),\\
&\dot p=-\partial_x H-\partial_u H(x,p,u) p, \qquad (x,p,u)\in T^*M\times\R.\\
&\dot u=\langle p,\partial_p H\rangle- H(x,p,u),
\enal
\right.
\eeq
The {\it Legendrian submanifolds}, defined by $n-$dimensional integrable manifolds w.r.t. \eqref{eq:ode}  have deep connection with contact topology and non-equilibrium thermodynamics, see \cite{A,BLN,EP,dL,L,R}. 
In particular, 
the generalized graph of a solution 
\[
{\rm Graph}(u):=\overline{\{(x, d_x u, u(x))\in T^*M\times\R| u(x)\text{ is differentiable at }x\in M\}}
\]
supplies us a Legendrian submanifold, although the solution $u$ is usually not smooth and  
${\rm Graph}(u)$ may be rather fragmental. Nonetheless, ${\rm Graph}(u)$ indeed contains a regular part (the so called {\it Mather set}), which dominates the asymptotic behaviors of local trajectories of \eqref{eq:ode} and further leads to the multiplicity of solutions and the complicated ordering structure among them.\medskip

In the following result, we show $\mathfrak M(u)$ supplies us an invariant set w.r.t. \eqref{eq:ode}, which is the so called {\it Mather set} and presents as a regular part of ${\rm Graph}(u)$. For this purpose, we need the so called {\it Tonelli Hamiltonian} in the weak KAM theory \cite{F}. Precisely, the Hamiltonian $H:(x,p,u)\in T^*M\times\R\rightarrow\R$ is requested to be $C^2-$smooth, superlinear and positive definite in $p\in T_x^*M$ for any $(x,u)\in M\times\R$.

\begin{thm}\label{thm:mat-set}
For a Tonelli Hamiltonian $H$, $c\in\mathfrak C$ fixed and a fixed solution $u(x)$ of \eqref{eq:hjc}, the Mather set 
\beq
\wt\cM(u):=\bigg\{\big(x,\partial_vL(x,v,u(x)),u(x)\big)\in T^*M\times\R\bigg|(x,v)\in \overline{\bigcup_{\mu\in\mathfrak M(u)}{\rm supp}(\mu)}\bigg\}
\eeq
is a compact invariant set w.r.t. equation \eqref{eq:ode} (associated with $H(x,p,u)-c$). Denote $\cM(u):=\pi\footnote{Here $\pi:TM\times\R ({\rm resp.}\; T^*M\times\R)\rightarrow M$ is the canonical projection.}\wt\cM(u)$ by the projected set onto $M$, then we actually have 
\beq\label{eq:graphic}
\wt\cM(u)=\{(x,d_xu(x),u(x))\in TM\times\R|x\in \cM(u)\}
\eeq
 and $\pi^{-1}:\cM(u)\rightarrow\wt\cM(u)$ is Lipschitz continuous. 
\end{thm}
\begin{rmk}
\begin{itemize}
\item[(i)] As a subset of $\wt\cM(u)$, we can define $\wt\cM_-(u)$ accordingly, which is also invariant w.r.t. equation \eqref{eq:ode}. In the sense of distribution, the multiplier $\partial_uL(x,v, u)$ is degenerate on $\wt\cM_-(u)$. That intrinsically causes the multiplicity of solutions. 
\item[(ii)] Under the same conditions as in Theorem \ref{thm:mat-set},  for any $\theta\in I(c)$, the Mather set (associated with $H(\cdot,\cdot,\theta):TM\rightarrow\R$) 
\beq
\wt\cM^\theta:=\bigg\{(x,\partial_vL(x,v))\in T^*M\bigg|(x,v)\in \overline{\bigcup_{\mu\in\mathfrak M^\theta}{\rm supp}(\mu)}\bigg\}
\eeq
is a Lipschitz graph over $\cM^\theta:=\pi\wt\cM^\theta$, which is invariant w.r.t. the Hamilton's equation associated with $H(\cdot,\cdot,\theta)$. That exactly coincides with the classical definition in the Aubry-Mather theory or weak KAM  
theory \cite{F,Mat}. Moreover, any subsolution $\om$ of 
\[
H(x,d_x\om,\theta)\leq c(\theta),\quad x\in M
\]
has to be a solution and differentiable on $\cM^\theta$, see \cite{DFIZ,FS}.
\end{itemize}
\end{rmk}

\subsection{Regularity of $c(\theta)$ and its constraint to Mather measures}

Since the definition of $\mathfrak M^\theta$ (resp. $\mathfrak M_-^\theta$) does not rely on solutions of \eqref{eq:hjc}, therefore, comparing to $\mathfrak M(u)$ (resp. $\mathfrak M_-(u)$) it has natural advantages in describing the structure of $\mathfrak C$:
 
\begin{thm}\label{thm:h4}
Assume {\bf H1-H4} for a continuous Hamiltonian $H$. The ergodic constant $c:\theta\in \R\rightarrow\R$ satisfies the followings:
\begin{itemize}
\item[1)] $c:\theta\in\R\rightarrow\R$ is locally Lipschitz, which is therefore differentiable for a.e. $\theta\in \R$;
\item[2)] If $\mathfrak M_-^\theta\neq\emptyset$, then the left hand derivative $c'_-(\theta)$ exists and equals $0$;
\item[3)] If $\mathfrak M_-^\theta=\emptyset$, then
$\varliminf_{\theta'\rightarrow\theta}\frac{c(\theta')-c(\theta)}{\theta'-\theta}>0$. Consequently, $c$ is strictly increasing at $\theta$.
\item[4)] If $c$ is differentiable at $\theta$, then 
\[
c'(\theta)=\max_{\mu\in\mathfrak M^\theta}\int_{TM}\partial_u H(x,v,\theta) d\mu(x,v)=\min_{\mu\in\mathfrak M^\theta}\int_{TM}\partial_u H(x,v,\theta) d\mu(x,v).
\]
Consequently, if $c'(\theta)>0$ then $\mathfrak M_-^\theta=\emptyset$ and $c'(\theta)=\int_{TM}\partial_u H(x,v,\theta) d\mu(x,v)$ for any $\mu\in\mathfrak M^\theta\backslash\mathfrak M_-^\theta$; If $c'(\theta)=0$, then $\mathfrak M^\theta\backslash\mathfrak M_-^\theta=\emptyset$. 
\item[5)] If $\varlimsup_{\theta'\rightarrow\theta_-}\frac{c(\theta')-c(\theta)}{\theta'-\theta}=0$, then $\varlimsup_{\theta'\rightarrow\theta_-}\mathfrak M^{\theta'}\subset\mathfrak M_-^\theta$;
\item[6)] If $\varlimsup_{\theta'\rightarrow\theta_-}\frac{c(\theta')-c(\theta)}{\theta'-\theta}>0$, then $\mathfrak M_-^\theta=\emptyset$;
\item[7)] If $\varlimsup_{\theta'\rightarrow\theta_+}\frac{c(\theta')-c(\theta)}{\theta'-\theta}=0$, then $\mathfrak M^\theta=\mathfrak M_-^\theta$;
\item[8)] If $\varlimsup_{\theta'\rightarrow\theta_+}\frac{c(\theta')-c(\theta)}{\theta'-\theta}>0$, then $\varlimsup_{\theta'\rightarrow\theta_+}\mathfrak M^{\theta'}\subset\mathfrak M^\theta\backslash\mathfrak M_-^\theta$.
\end{itemize}
\end{thm}

\begin{cor}\label{cor:h5-c}
Assume {\bf H1-H4} and\smallskip

\noindent{\bf H5':} $H(x,p,u)$ is convex in $(p,u)$ for any $x\in M$.

Then the followings hold: 
\begin{itemize}
\item[1)] For the case $c_0:=\min\mathfrak C$ is finite, $\mathfrak M^\theta_-\neq \emptyset$ for any $\theta\in c^{-1}(c_0)$.
\item[2)] If $c(\theta)\in(c_0,+\infty)$ with $c_0:=\inf\mathfrak C$ (could be $-\infty$), then $\mathfrak M_-^{\theta}=\emptyset$.
\item[3)] For any $\theta\in\R$, if $\mathfrak M_-^\theta\neq\emptyset$, then $c(\theta)=\min\mathfrak C$ is finite; 
\item[4)] If $\mathfrak M_-^\theta=\emptyset$, then $c(\theta)>\inf\mathfrak C$.
\end{itemize}
\end{cor}

\subsection*{Organization of the article} The paper is organized as follows: In Section \ref{s2}, we present a variational interpretation to the viscosity solutions of \eqref{eq:hjc}, then prove Theorem \ref{thm:l-cp} and Theorem \ref{thm:g-cp}. In Section \ref{s3}, we supply a list of examples to show the optimality of our assumptions in the main conclusions. In Section \ref{s4}, we prove the invariance of the Mather set w.r.t. the contact Hamilton's equation \eqref{eq:ode}. In Section \ref{s5}, we discuss the micro structure of $\mathfrak C$ and prove Theorem \ref{thm:h4} and Corollary \ref{cor:h5-c}. 

\subsection*{Acknowledgements} 
This work is supported by the National Key R\&D Program of China (No. 2022YFA1007500) and the National Natural Science Foundation of China (No. 12231010). The author is grateful to Ya-Nan Wang for numerical assistance in  the examples. 

\section{Variational interpretation of  solutions and the Comparison Principle}\label{s2}

\subsection{Weak KAM solutions of classical H-J equations}\label{s2.1}

This part is a brief introduction of the weak KAM theory. We assume the classical Hamiltonian $H:(x,p)\in T^*M\rightarrow\R$ is superlinear and convex in $p$ (for any contact Hamiltonian, by assigning a constant $a\in\R$ to the third component we instantly get such a Hamiltonian $H(\cdot,\cdot,a):T^*M\rightarrow\R$). The Fenchel's transformation gives us the Lagrangian $L:(x,v)\in TM\rightarrow\R$ by 
\[
L(x,v):=\max_{p\in T^*_xM}\{\langle v,p\rangle-H(x,p)\}.
\]
As in \eqref{eq:e-const}, a unique ergodic constant $c(H)\in\R$ (also called {\it Ma\~n\'e's critical value}) exists such that 
\beq\label{eq:hj}
H(x,d_x\om)=c(H),\quad x\in M
\eeq
has a solution.

\begin{defn}
A function $\om\in C(M,\R)$ is called a {\bf weak KAM solution} of \eqref{eq:hj}, if it satisfies:
\begin{itemize}
\item[1)] $\om\prec L+c(H)$, i.e. for any $x,y\in M$, $T\geq 0$ and absolutely continuous $\gamma:[0,T]\rightarrow M$ connecting $x$ to $y$, $\om(y)-\om(x)\leq\int_0^TL(\gamma,\dot\gamma)+c(H) dt$;
\item[2)] For any $x\in M$, there exists an absolutely continuous curve $\gamma:(-\infty, 0]\rightarrow M$ ending with it, such that 
\[
\om(x)-\om(\gamma(t))=\int_t^0L(\gamma(s),\dot\gamma(s))+c(H) ds,\quad\forall\, t\geq 0.
\]
Such a curve is called a {\bf backward calibrated curve} w.r.t. $\om$.
\end{itemize}
\end{defn}

\begin{lem}\cite{CCIZ,F}\label{lem:sub}
1). $\om$ is a subsolution of \eqref{eq:hj} iff $\om(x)\prec L+c(H)$; 2) $\om$ is a viscosity solution of \eqref{eq:hj} iff $\om$ is a weak KAM solution.
\end{lem}

\begin{defn}[Mather measure \cite{Mn}]\label{defn:mat}
Let $\cP(TM,\R)$ be the space of all Borel probability measures on $TM$, then $\mu\in\cP(TM,\R)$ is called {\bf closed} if it satisfies:
\begin{itemize}
\item[1)] $\int_{TM}|v|_x d\mu(x,v)<+\infty$;
\item[2)] $\int_{TM}\langle v, d_x\phi\rangle d\mu(x,v)=0$ for any $\phi\in C^1(M,\R)$.
\end{itemize}
Denote by $\cC$ the set of all closed measures, then $\inf_{\mu\in\cC}\int_{TM} L(x,v)+c(H)d\mu\geq 0$. Consequently, any $\varrho\in\cC$ is called a {\bf Mather measure}, if $\int_{TM} L(x,v)+c(H)d\varrho=0$. Usually we denote by $\mathfrak M$ the set of all Mather measures.
\end{defn}

\begin{thm}\label{thm:mat-graph}
\begin{itemize}
\item \cite{Mat} If $H:T^*M\rightarrow \R$ is a $C^2-$smooth Tonelli Hamiltonian, then the {\bf Mather set} defined by 
\beq\label{eq:mat-set}
\wt\cM:=\bigg\{
\big(x,\partial_vL(x,v)\big)\in T^*M\bigg|(x,v)\in 
\overline{\bigcup_{\mu\in\mathfrak M}{\rm supp}(\mu)}\bigg\}
\eeq
is invariant with respect to the following {\bf Hamilton's equation}:
\beq\label{eq:h-ode}
\left\{
\beal
&\dot x=\partial_p H(x,p),\\
&\dot p=-\partial_x H(x,p).\
\enal
\right.
\eeq
Furthermore, $\pi^{-1}:\cM:=\pi(\wt\cM)\rightarrow \wt\cM$ is a Lipschitz diffeomorphism.
\item \cite{F}
 any solution $u(x)$ of \eqref{eq:hj} is differentiable on $\cM$, with $\wt\cM=\{(x,d_x u(x)):x\in\cM\}$.
 \end{itemize}
\end{thm}

\subsection{Variational interpretations of solutions for contact H-J equations}\label{s2.2}

As the preliminary, we first prove Proposition \ref{prop:mane} here (some parts can be found from \cite{Tr}, but for the consistency we prove here individually):\medskip

\noindent{\it Proof of Proposition \ref{prop:mane}:}

1). Due to {\bf H3}, for any $a<b$ 
\[
H(x,d_x u_b,a)\leq H(x,d_x u_b,b)\leq c(b),\quad {\rm a.e. } \;x\in M
\]
where $u_b$ is a solution of $H(\cdot,\cdot,b)$, therefore, $c(a)$ can't be greater than $c(b)$ since $c(a)$ has to be the minimal value in $\R$ such that $H(\cdot,\cdot,a)$
 has a subsolution. For any  sequence $\R\ni a_n\to a$ as $ n\rightarrow+\infty$, we can find a sequence of solutions $\{u_n\}_{n\in\N}$ with $u_n(0)\equiv 0$ for any $n\in\N$. Due to Lemma \ref{lem:sub}, $\{u_n\}_{n\in\N}$ are uniformly Lipschitz, then uniformly bounded. Suppose $c_*$ is an accumulating point of $c(a_n)$ as $n\rightarrow+\infty$, then the associated subsequence of $\{u_n\}_{n\in\N}$ also has an accumulating function $\om$ as $n\rightarrow +\infty$, which is exactly a solution of  
\[
H(x, d_x\om, a)=c_*,\quad x\in M.
\]
Due to \cite{FS}, such a $c_*\in\R$ is unique, so $c_*=c(a)$ and $\lim_{n\rightarrow+\infty} c(a_n)=c(a)$ follows. 

2). For any $c\in\mathfrak C$, there must exist a $a\in\R$ and a solution $u_c$ of 
\[
H(x, d_xu_c,a)=c,\quad x\in M.
\]
Observe that $u_c-\max_Mu_c+a$ (resp. $u_c-\min_Mu_c+a$) has to be a subsolution (resp. supersolution) of \eqref{eq:hjc}, Due to the {\it Perron's method} (see \cite{Is} for instance), there must be a solution $u$ of \eqref{eq:hjc} such that $u_c-\max_Mu_c+a\leq u\leq u_c-\min_Mu_c+a$.
 \qed\medskip

For any $c\in\mathfrak C$, we can always find a solution $u\in C(M,\R)$ of the associated \eqref{eq:hjc}. Denote by 
\[
{\bf L}^u(x,v):=L(x,v, u(x))\quad {\rm (\, resp.}\; {\bf H}^u(x,p):=H(x,p, u(x)) \,{\rm )}
\]
 for any $(x,v)\in TM$ (resp. $(x,p)\in T^*M$), then $u\prec{\bf L}^u+c$ due to item 1 of Lemma \ref{lem:sub}, which further indicates $u\in {\rm Lip}(M,\R)$. Moreover, $c$ is also the ergodic constant of \eqref{eq:hj} associated with the Hamiltonian ${\bf H}^u(x,p)$. So we conclude:
\begin{prop}[Dynamical Programming]\label{prop:dp}
Assume {\bf H1-H3}. Suppose $u\in C(M,\R)$ is a solution of \eqref{eq:hjc}, then 
\beq\label{eq:dp}
u(x)=\inf_{\substack{\gamma\in C^{ac}([T,0],M)\\\gamma(0)=x}}\Big(u(\gamma(T))+\int_T^0L(\gamma,\dot\gamma, u(\gamma))+c dt\Big),\quad\forall\, T\leq 0.
\eeq
Moreover, there exists a Lipschitz  curve $\gamma_x:[-\infty,0]\rightarrow M$ ending with $x$, such that 
\[
u(x)-\ u(\gamma_x(T))=\int_T^0L(\gamma_x,\dot\gamma_x, u(\gamma_x))+c dt,\quad \forall\, T\leq 0.
\]
\end{prop}

\subsection{Two types of Mather measures}\label{s2.3}

As we know, any  solution $u$ of \eqref{eq:hjc} has to be a weak KAM solution associated with ${\bf H}^u$. Then for any $x\in M$, there exists a backward calibrated (w.r.t. $u$) curve $\gamma_x:(-\infty,0]\rightarrow M$ ending with it, such that 
\[
u(\gamma_x(b))-\ u(\gamma_x(a))=\int_a^bL(\gamma_x,\dot\gamma_x, u(\gamma_x))+c \ dt,\quad \forall\, a\leq b\leq 0.
\]
Due to  Proposition \ref{prop:dp}, such a curve $\gamma_x$ is Lipschitz. Furthermore, for any $a< b\leq 0$, we can define a probability measure $\mu_x^{a,b}$ by 
\[
\int_{TM} f(x,v) d\mu_x^{a,b}(x,v):=\frac{1}{b-a}\int_a^b f(\gamma_x(t),\dot\gamma_x(t)) dt,\quad\forall\, f\in C_c(TM,\R).
\]
As $b-a\rightarrow+\infty$, any accumulating measure $\mu$ (in the sense of weak topology) of $\mu_x^{a,b}$ has to be a closed measure and 
\[
\int_{TM} L(x,v, u(x)) +c\ d\mu(x,v)=0.
\]
Denote by $\mathfrak M(u)$ the set of all such accumulating measures, then it's actually a subset of the Mather measures associated with ${\bf H}^u$ (see Definition \ref{defn:mat}) .\medskip

On the other side, for any $c\in\mathfrak C$, there exists a connected interval $I(c)$, such that $c$ is the Ma\~n\'e critical value of ${\bf H}^\theta$ for any $\theta\in I(c)$. 
Denote by 
\[
\mathfrak M^\theta:=
\{\text{all Mather measures associated with }{\bf H}^\theta\},
\]
 the following conclusion can be easily drawn:
\begin{lem}\label{lem:mono-mea}
Assume {\bf H1-H3}. For any $\theta_1\leq \theta_2$ contained in $I(c)$, any Mather measure $\mu\in\mathfrak M^{\theta_1}$ has to be contained in  $\mathfrak M^{\theta_2}$.
\end{lem}
\proof
Due to {\bf H3}, 
\[
0=\int_{TM}[L(x,v,\theta_1)+c]d\mu(x,v)\geq \int_{TM}[L(x,v,\theta_2)+c]d\mu(x,v)\geq 0
\]
which urges $\mu$ to be a Mather measure associated with ${\bf H}^{\theta_2}$. Consequently,  for any $(x,v)\in {\rm supp}(\mu)$ with $\mu\in\mathfrak M^{\theta_1}$, we have $L(x,v,\theta_1)=L(x,v,\theta_2)$ once $\theta_2\geq \theta_1$ and is contained in $I(c)$.
\qed
\begin{rmk}
Theoretically, $\mathfrak M^{\sup I(c)}$ contains the most Mather measures than any other $\mathfrak M^{\theta\in I(c)}$, so we can denote by $
\mathfrak M(c):=\mathfrak M^{\sup I(c)}$. Such a set does not depend on any information of solutions of \eqref{eq:hjc}, whereas $\mathfrak M(u)$ does. 
\end{rmk}

\begin{defn}[ordinal Mather measure]\label{defn:t-mea}
Suppose $u$ is a solution of \eqref{eq:hjc} and $\partial_u L(x,v,u(x))$ exists for any $(x,v)\in{\rm supp}(\mu)$ with $\mu\in\mathfrak M(u)$, then we denote by 
\beq\label{eq:u-tint}
\mathfrak M_-(u):=\Big\{\mu\in\mathfrak M(u)\Big|\int_{TM}\partial_uL(x,v, u(x))d\mu(x,v)=  0\Big\}.
\eeq
Similarly, suppose $\partial_u L(x,v,\theta)$ exists for some $\theta\in I(c)$ and $(x,v)\in{\rm supp}(\mu)$ with $\mu\in\mathfrak M(c)$, we denote by 
\beq\label{eq:c-tint}
\mathfrak M_-^\theta:=\Big\{\mu\in\mathfrak M^\theta\Big|\int_{TM}\partial_uL(x,v, \theta)d\mu(x,v)= 0\Big\}.
\eeq
\end{defn}
As a complement, we point out in above definition the existence of $\partial_u L$ in \eqref{eq:u-tint} and \eqref{eq:c-tint} can be guaranteed by {\bf H4} and {\bf H4'} respectively. A standard argument based on the Fenchel's conjugacy between $L$ and $H$ can prove that, see e.g. Lemma 4.1 of \cite{Ch}.

\subsection{Proof of the comparison principle} At first, we propose the following technical Lemma which will be used in the proof of Theorem \ref{thm:l-cp} and Theorem \ref{thm:g-cp}. Such a Lemma was originally proved in Theorem 4.1 of \cite{CFZZ}> For the consistency of the paper, we briefly sketch the proof for it.

\begin{lem}\label{lem:lya}
Suppose that $u:M\to\R $ is a subsolution of 
$${\bf H}^\varphi(x,d_xu):=H(x,d_x u,\varphi(x))=c$$
and $w:M\to\R $ is a solution of 
$${\bf H}^\psi(x,d_xw):=H(x,d_x w,\psi(x))=c,$$
for certain continuous functions $\varphi$ and $ \psi$,  then one of the following two holds:
\begin{enumerate}
\item[(1)] The maximum of $u-w$ can be attained at a point $x_{\max}\in M$ where $\varphi(x_{\max})-\psi(x_{\max})\leq 0$.
\item[(2)] We can find a Lipschitz curve $\gamma : t\in (-\infty,0]\to M$ which is both $u$-calibrated for 
${\bf L}^\varphi+c$ and $w$-calibrated for ${\bf L}^\psi+c$, such that for all $t\in(-\infty,0]$,
\begin{align*}
u(\gamma(t))- w(\gamma(t))=\max(u-w), \qquad \varphi(\gamma(t))-\psi(\gamma(t))>0.
\end{align*}
\end{enumerate}
Moreover,  if (1) does not hold, denoting by $K<+\infty$ a Lipschitz constant 
for the Lipschitz curve $\gamma:(-\infty,0]\to M$ obtained in (2),
we can find
a closed measure $ \mu\in\cC$ of which ${\rm supp} (\mu)$ is contained in the compact subset $\{(x,v)\in TM \mid | v|_x\leq K\}$,
such that
$$\int_{TM}\big[ L(x,v, \varphi(x))+c\big]\,d\mu(x,v)=\int_{TM} \big[L(x,v, \psi(x))+c\big]\,d\mu(x,v)=0$$
and for all $(x,v)\in {\rm supp} (\mu)$,
$$u(x)-w(x)=\max_M(u-w),\ \  \varphi(x)-\psi(x) > 0.$$
 \end{lem}
 
\begin{proof}
Pick $x_0\in M$ such that
$u(x_0)- w(x_0)=\max_M(u- w)$. If $\varphi(x_0)-\psi(x_0)\leq 0$, then alternative (1) holds.

Suppose now that $\varphi(x_0)-\psi(x_0)> 0$. Since $w$ is a solution of 
${\bf H}^\psi(x,d_xw)=c$, we can find a w-calibrated curve  $\gamma:\,(-\infty,0]\to M$ ending with $x_0$, which is Lipschitz with a constant $K<+\infty$ due to Lemma \ref{lem:sub}.  In particular,
\begin{equation}\label{inftycalibrationv}
w(x_0)-w(\gamma(t))=
\int_t^0L(\gamma(s),\dot\gamma(s), \psi(\gamma(s)))+c\,d s,\;\;\forall \; t\leq 0.
\end{equation}
Since $\varphi(x_0)-\psi(x_0)> 0$, we can define $t_0\in (-\infty,0]\cup\{-\infty\}$ by
\begin{equation}\label{deft0}t_0=\inf\{t\leq 0\mid \varphi(\gamma(s))-\psi(\gamma(s))>0,\text{ for all $s\in[t,0]$}\}.
\end{equation}
By continuity of $\varphi\circ \gamma-\psi\circ\gamma$, there hold $t_0<0$ and $(\varphi-\psi)(\gamma(t_0))=0$ if $t_0$ is finite.

\underline{\bf Claim:} We have
\begin{equation}\label{Constance1gen}
u(x_0)- w(x_0)=u(\gamma(t))- w(\gamma(t)),\quad\forall \,t\in(t_0,0],
\end{equation}
and $\gamma: (t_0,0] \to M$ is not only 
 $w$-calibrated but also $u$-calibrated.
 Moreover
\begin{equation}\label{Constance2gen}
L\bigl(\gamma(t),\dot\gamma(t), \varphi(\gamma (t))\bigr)=L\bigl(\gamma(t),\dot\gamma(t), \psi(\gamma (t))\bigr),
\quad {\rm a.e.}\; t\in(t_0,0].
\end{equation}

We now prove the Claim.
In fact by condition {\bf H3}, we obtain
\begin{equation}\label{Momogen}
L\bigl(\gamma(s),\dot\gamma(s), \varphi(\gamma (s))\bigr)\leq L\bigl(\gamma(s),\dot\gamma(s), \psi(\gamma (s))\bigr),\quad {\rm a.e.}\;
s\in(t_0,0].
\end{equation}
Hence for any $t\in (t_0,0]$, we get 
\be\label{BunchInegen}
 u(x_0)-u(\gamma(t))&\leq& \int_{t}^0\big[L\bigl(\gamma(s),\dot\gamma(s), \varphi(\gamma (s))\bigr)+c\big]\,d s\nonumber\\
&\leq& \int_{t}^0\big[L\bigl(\gamma(s),\dot\gamma(s), \psi(\gamma (s))\bigr)+c\big]\,d s\\ 
&=& w(x_0)-w(\gamma(t)).\nonumber
\ee
Recall that  $u(x_0)- w(x_0)=\max_M(u- w)$, which implies 
 $ u(x_0)-u(\gamma(t))=w(x_0)-w(\gamma(t))$ and forces all inequalities
 in \eqref{BunchInegen} to be equality. Consequently, the curve $\gamma: (t_0,0] \to M$   
 is also $u$-calibrated and
 \begin{equation*}
 w(x_0)-w(\gamma(t))=u(x_0)-u(\gamma(t)), \, \textup{for all $t\in(t_0,0]$},
 \end{equation*}
so it verifies equality \eqref{Constance1gen} then \eqref{Constance2gen}, 
which finally establishes the Claim.\medskip

$\bullet$ If $t_0$ is finite, then by continuity $\max_M(u- w)=u(x_0)- w(x_0)=u(\gamma(t_0))- w(\gamma(t_0))$
and $(\varphi-\psi)(\gamma(t_0))=0$ due to \eqref{deft0}. Alternative (1) holds in that case.

$\bullet$ If $t_0=-\infty$, then the curve 
$\gamma: (-\infty,0]\to M$ satisfies alternative (2). If so, the time average w.r.t. $t>0$ gives us a probability measure $\mu_t$ by 
$$\int_{TM}f(x,v)\,d \mu_t:=\frac1t\int_{-t}^0f(\gamma(s),\dot\gamma(s))\,d s,\quad  \forall\;f\in C_c(TM,\R).$$
Since $ \gamma$ is $K-$Lipschitz,  supp$( \mu_t)\subset \{(x,v)\in TM: |v|\leq K\}$. So we can choose a sequence $t_n\to +\infty$ such that  
$ \mu_{t_n}\to \mu$ in the weak topology. Since $M$ is compact and $\gamma$ is both $u-$calibrated and $w-$calibrated, we can easily prove that  $\mu$ is a closed measure and satisfies
\ben
	\int_{TM} \big[L(x,v, \psi(x))+c\big]\,d\mu(x,v)
=\int_{TM} \big[L(x,v, \varphi(x))+c\big]\,d\mu(x,v)=0.
\een
Since any $(x,v)\in{\rm supp}(\mu)$ should be a $\alpha-$limit point of $\gamma$, then 
$u(x)-w(x)=\max_M(u-w)$ and $\varphi(x)-\psi(x)>0$ away from alternative (1). This finally finishes the proof. 
\end{proof}

\bigskip

\noindent{\it Proof of Theorem \ref{thm:l-cp}:}
1). Due to Lemma \ref{lem:lya}, we conclude that either $u_1-u_2$ attains the maximum at a point $x_{\max}$ such that $u_1(x_{\max})\leq u_2(x_{\max})$, or $u_1-u_2$ attains a positive maximum on the set $\pi({\rm supp}(\mu))\subset M$ for some closed measure  $\mu$ such that 
\[
\int_{TM}\big[ L(x,v, u_1(x))+c\big]\,d\mu(x,v)=\int_{TM} \big[L(x,v, u_2(x))+c\big]\,d\mu(x,v)=0.
\]
 Since alternative 1 coincides with our assertion, so it remains to consider alternative 2.
Due to {\bf H3}, $L(x,v,r)=L(x,v, u_1(x))$ for all $(x,v)\in{\rm supp}(\mu)$ and $r\in[u_2(x), u_1(x)]$. Moreover, from the construction of $\mu$ in Lemma \ref{lem:lya}, $\mu\in\mathfrak M(u_1)\cap\mathfrak M(u_2)$. Therefore, due to {\bf H4}, there holds
\[
\partial_uL(x,v, r)=0,\quad\forall\,(x,v)\in{\rm supp}(\mu),\; r\in[u_2(x), u_1(x)]
\]
which further indicates $\mu\in\mathfrak M_-(u_1)$ and then \eqref{eq:cp-mea-2}. However, that contradicts with the premise of alternative 2, i.e., $u_1-u_2>0$ on $\pi({\rm supp}(\mu))$, so our assertion follows. 

2). As a corollary of 1), for any other solution $\wt u$ of \eqref{eq:hjc} different from $u$, $\mathfrak M_-(u)=\emptyset$ ensures that $u\leq \wt u$, so $u$ has to be the minimal solution. If $\wt u-u>0$ at some point $x_0\in M$, then due to the proof of 1), we conclude that $\wt u-u$ attains the positive maximum on the set supp$(\mu)$, where $\mu\in\mathfrak M(\wt u)\cap\mathfrak M(u)$. Therefore, $\mu\in\mathfrak M_-(u)$ which contradicts with the assumption $\mathfrak M_-(u)=\emptyset$. That implies $\wt u=u$.
\qed
\bigskip

\noindent{\it Proof of Theorem \ref{thm:g-cp}:}
1). 
Similarly with Theorem \ref{thm:l-cp}, we directly assume that 
$u_1-u_2$ attains a positive maximum on the set $\pi({\rm supp}(\mu))\subset M$ for some closed measure  $\mu$ satisfying
\[
\int_{TM}\big[ L(x,v, u_1(x))+c\big]\,d\mu(x,v)=\int_{TM} \big[L(x,v, u_2(x))+c\big]\,d\mu(x,v)=0.
\]
If so,
 \[
 L(x,v,r)=L(x,v, u_1(x)),\quad\forall\; (x,v)\in{\rm supp}(\mu), \; r\in [u_2(x), u_1(x)]\subset\R
 \]
 due to {\bf H3}. Furthermore, due to {\bf H5} we know that $L(x,v,u)$ is concave in $u$, so 
 \[
 \max_{u\in\R} L(x,v,u)=L(x,v, u_1(x))=L(x,v,r),\quad \forall\; (x,v)\in{\rm supp}(\mu),\; r\leq u_1(x).
 \]
Consequently, we can find a $r_c<\min \{\theta,\min_{x\in M} u_1(x)\} $ such that 
\ben
\int_{TM}\big[ L(x,v, r_c)+c\big]\,d\mu(x,v)=\int_{TM}\big[ L(x,v, \theta)+c\big]\,d\mu(x,v)=0
\een
for any fixed $\theta\in I(c)$, which indicates $\mu\in\mathfrak M^\theta$. However, due to {\bf H4'} previous equality implies that 
\[
\partial_u L(x,v,\theta)=0,\quad \forall\;  (x,v)\in{\rm supp}(\mu),
\]
so $\mu\in\mathfrak M_-^\theta$ which contradicts with \eqref{eq:cp-mea}.


2). As a corollary of 1),  $\mathfrak M_-^\theta=\emptyset$ implies that any two solutions satisfy $u_1\leq u_2$ and $u_2\leq u_1$ simultaneously, so $u_1=u_2$. 

3) If $\mathfrak M_-^\theta=\{\mu\}$ is a singleton, then for any two solutions $u_1, u_2$, their integations with respect to $\mu$ are comparable. Due to 1), $u_1$ and $u_2$ are also comparable. Moreover, if the integrations equal to each other, then $u_1=u_2$.
 \qed

\section{Examples: the necessity of integrable non-degeneracy}\label{s3}

For classical Hamilton-Jacobi equations independent of $u$,  similar results as our Theorem \ref{thm:l-cp} and Theorem \ref{thm:g-cp} were proved earlier in Theorem 12.6 of \cite{F2}  (see also \cite{DS,I} for different versions). As a generalization to the nonlinearly $u-$dependent case, our conclusions also reveal additional features, which would never happen 
for the $u-$independent occasion. In this section, we will show these new features by examples. Besides, such examples show certain optimality of our assumptions. 

\medskip

The first example is devoted to show that, if $\mathfrak M_-^\theta\neq\emptyset$, more than one solution of \eqref{eq:hjc} can be found. Therefore, criterion \eqref{eq:cp-mea} in Theorem \ref{thm:g-cp} is necessary.
\begin{ex}
Suppose $H_0(x,p)=p^2/2+\cos2\pi x-1\in\R$ and $H(x,p,u)= H_0(x,p)+(1-\cos2\pi x) u$ with $(x,p,u) \in T^*(\R\slash\Z)\times\R$, then $0$ is the Ma\~n\'e's critical value of $H_0$ system. Recall that any viscosity solution $u(x)$ of \eqref{eq:hjc} has to be a weak KAM solution, vice versa (due to Lemma \ref{lem:sub}). Besides, $H(x,p,u)=H(1-x,p,u)$ for any $(x,p,u)\in  T^*(\R\slash\Z)\times\R$. Consequently, 
\beq
u_\lb(x)=\left\{
\begin{aligned}
&1-(\frac{\cos\pi x}\pi+\lb)^2,\quad {\rm when}\; x\in[0,1/2),\\
&1-(\lb-\frac{\cos\pi x}\pi)^2,\quad {\rm when}\; x\in[1/2,1),\\
\end{aligned}
\right.
\eeq
is always a viscosity solution for any $\lb\in [0,+\infty)$, so \eqref{eq:hjc} has infinitely many solutions (see Fig. \ref{fig1}). 

On the other side, for $\lb_1>\lb_2\geq 0$, the associated solution $u_{\lb_1}> u_{\lb_2}$. This is because the Dirac measure $\dt_{(0,0)}$ supported on $(0,0)\in T(\R\slash\Z)$ is the only measure in $ \mathfrak M^{\theta=0}$. 
\end{ex}
\begin{figure}
\begin{center}
\includegraphics[width=8cm]{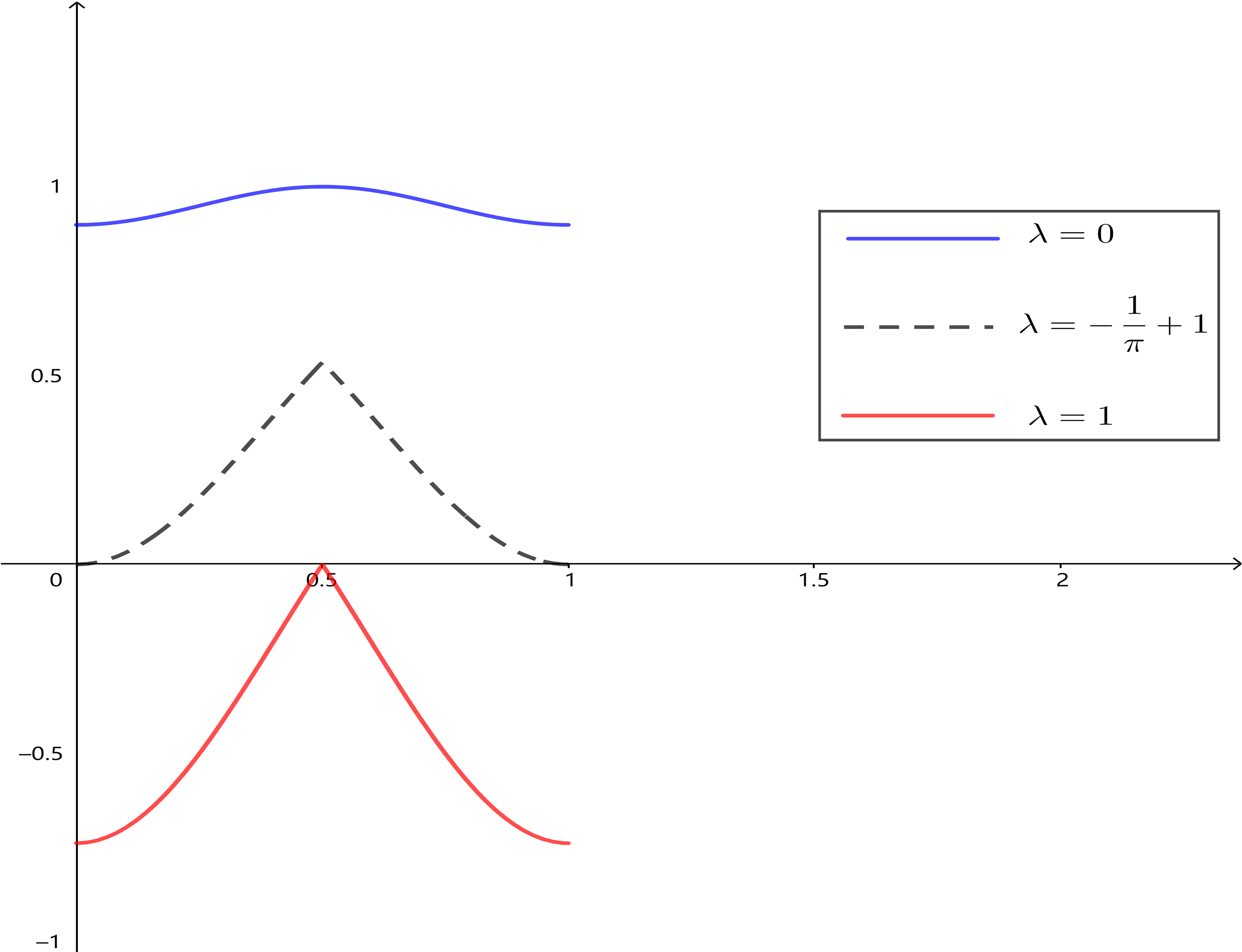}
\caption{When $\lb=0$, the solution $u_\lb$ is actually $C^1-$smooth, which therefore is a classical solution. As $\lb\geq 0$ increases, $u_\lb$ descends in the way comforming to item 3) of Theorem \ref{thm:g-cp}.}
\label{fig1}
\end{center}
\end{figure}

The second example is used to show that $\emptyset\neq\mathfrak M_-^\theta\subsetneq\mathfrak M^\theta$ could happen, which causes more complicated multiplicity of solutions.

\begin{ex}
Suppose $\alpha:\R\slash 2\Z\rightarrow\R$ is defined by 
\beq
\alpha(x)=\left\{
\begin{aligned}
&1+\cos2\pi x,\quad &{\rm when}\;& x\in[1/2,3/2),\\
&0,\quad &{\rm when}\;& x\in[0,1/2)\cup[3/2,2),\\
\end{aligned}
\right.
\eeq
then $H(x,p,u):=H_0(x,p)+\alpha(x)u$ is $2-$periodic of $x$, where $H_0(x,p)$ is as in previous example. For such a $H(x,p,u)$, the associated \eqref{eq:hjc} with $c=0$ admits two $C^1-$smooth solutions (see Fig. \ref{fig2}) expressed by 
\beq
u_1(x)=\left\{
\begin{aligned}
&g_1( 1/2)+\int_{1/2}^x\sqrt{2(1- \cos2\pi s)}\, {\rm ds}, \quad &{\rm when}\; x\in[0,1/2),\\
&g_1(x),\quad &{\rm when}\; x\in[1/2,1),\\
&g_1(2-x),\quad &{\rm when}\; x\in[1,3/2),\\
&g_1(1/2)+\int^{2-x}_{1/2}\sqrt{2(1- \cos2\pi s)}\, {\rm ds},\quad &{\rm when}\; x\in[3/2,2),\\
\end{aligned}
\right.
\eeq
\beq
u_2(x)=\left\{
\begin{aligned}
&g_2(1/2)+\int_x^{1/2}\sqrt{2(1- \cos2\pi s)}\, {\rm ds},\quad &{\rm when}\; x\in[0,1/2),\\
&g_2(x),\quad &{\rm when}\; x\in[1/2,1),\\
&g_2(2-x),\quad &{\rm when}\; x\in[1,3/2),\\
&g_2(1/2)+\int_{2-x}^{1/2}\sqrt{2(1- \cos2\pi s)}\, {\rm ds},\quad &{\rm when}\; x\in[3/2,2),\\
\end{aligned}
\right.
\eeq
where $g_1(x)$ (resp. $g_2(x)$) is the solution of 
\[
\left\{
\begin{aligned}
&g'(x)=\sqrt{2(1-\cos2\pi x)-2(1+\cos2\pi x) g}\\
&g(1)=0
\end{aligned}
\right.
\]
\[
\bigg(\,{\rm resp.}\quad \left\{
\begin{aligned}
&g'(x)=\sqrt{2(1-\cos2\pi x)-2(1+\cos2\pi x) g}\\
&g(1)=0
\end{aligned}
\right.\bigg).
\]
In this case, $\mathfrak M_-^{\theta=0}=\{\dt_{(0,0)}\}\subsetneq \mathfrak M^{\theta=0}=\{\dt_{(0,0)}, \dt_{(1,0)}\}$. 
\end{ex}

\begin{figure}
\begin{center}
\includegraphics[width=9cm]{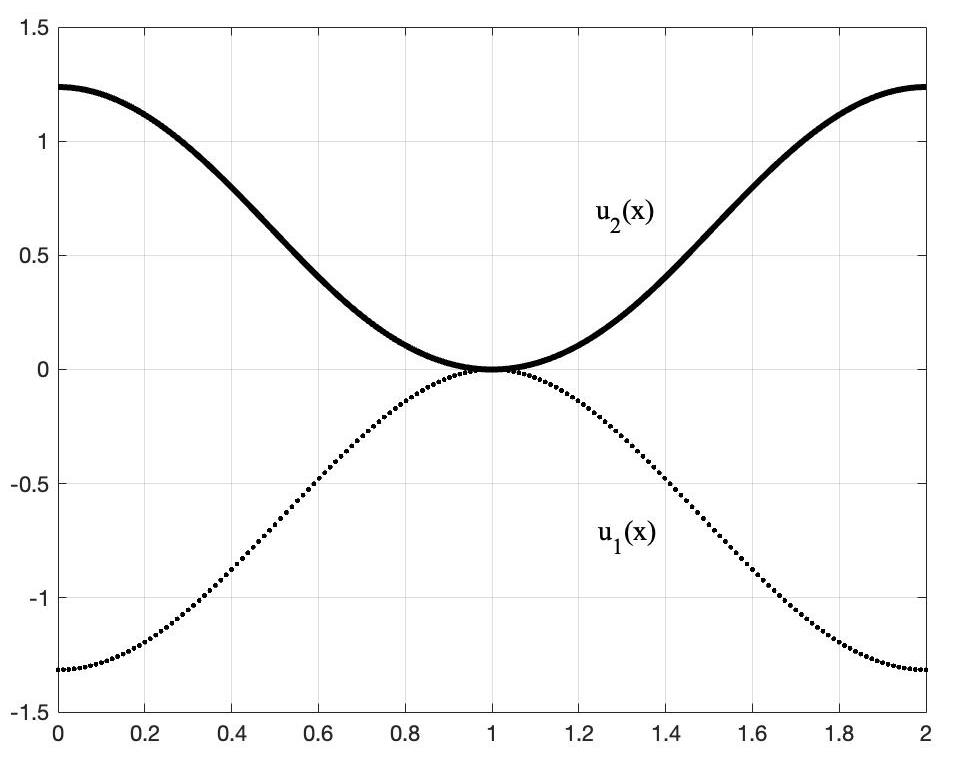}
\caption{A couple of classical solutions exists once $0\neq\mathfrak M_-^\theta\subsetneq\mathfrak M^\theta$.}
\label{fig2}
\end{center}
\end{figure}

\section{Contact dynamics of the Mather measures}\label{s4}

Throughout this section, the Hamiltonian $H:T^*M\times\R\rightarrow\R$ is assumed to be Tonelli and $C^2$ smooth. Since $c\in\mathfrak C$ is fixed, we can denote by $
H_c(x,p,u):=H(x,p,u)-c$ 
and the associated Lagrangian $L_c:=L(x,v,u)+c$. For convenience, we rewrite equation \eqref{eq:ode} associated with $H_c$ here:
\beq
\left\{
\beal
&\dot x=\partial_p H(x,p,u),\\
&\dot p=-\partial_x H-\partial_u H(x,p,u) p, \qquad (x,p,u)\in T^*M\times\R.\\
&\dot u=\langle p,\partial_p H\rangle- H(x,p,u)+c,
\enal
\right.
\eeq
of which $\Phi_{H,c}^t:(x,p,u)\in T^*M\times\R\rightarrow T^*M\times\R$ is the flow map. The {\it Legendre transform} 
\[
\cL: (x,p,u)\in T^*M\times\R\rightarrow  (x,\partial_p H(x,p,u), u)\in TM\times\R
\]
transports the flow $\Phi_{H,c}^t$ into the corresponding {\it Lagrangian flow} $\Phi_{L,c}^c:=\cL\circ\Phi_{H,c}^t$ which solves the following  {\it contact Euler-Lagrange equation}
\beq\label{eq:e-l}
\left\{
\beal
&\dot x=v,\\
&\frac{d}{dt}\partial_vL(x,v,u)=\partial_x L(x,v,u)+\partial_u L(x,v,u)\partial_vL(x,v,u) , \qquad (x,v,u)\in TM\times\R.\\
&\dot u=L(x,v,u)+c.
\enal
\right.
\eeq
We will show that $\mu\in\mathfrak M(u)$ is invariant w.r.t. $\Phi_{L,c}^t$. In other words,  ${\rm supp}(\mu)$ is the $\alpha-$limit set of a backward calibrated trajectory (w.r.t. $u$) $\gamma:(-\infty,0]\rightarrow M$ which actually induces a real trajectory of \eqref{eq:e-l} by the following Lemma:
\begin{lem}\label{lem:cali-flow}
Any backward $u-$calibrated curve $\gamma:(-\infty,0]\rightarrow M$  is $C^1$ smooth. Moreover, $\Big(\gamma(t), \partial_vL\big(\gamma(t),\dot\gamma(t), u(\gamma(t))\big), u(\gamma(t))\Big):t\in(-\infty,0]\rightarrow T^*M\times\R$ is a $C^1$ trajectory of \eqref{eq:ode}.
\end{lem}
\proof
The scheme is quite similar with the one in Theorem 3.6.1 of \cite{F}. First, due to Proposition \ref{prop:dp}, there holds 
\[
\frac d{dt} u(\gamma(t))= L(\gamma(t), \dot\gamma(t), u(\gamma(t)))+c,\quad {\rm a.e.} \,t\in(-\infty,0).
\]
Suppose $s<0$ is a differentiable point of $\gamma$, i.e., 
\[
(x_0,p_0,u_0):=\big(\gamma(s), \partial_ vL\big(\gamma(s),\dot\gamma(s), u\circ \gamma(s)\big), u\circ\gamma(s)\big)\in T^*M\times\R
\]
is uniquely identified, then for suitably small $\dt, P_s>0$, the flow $\Phi_{H,c}^\tau$ of \eqref{eq:ode} (associated with $H_c$) is well defined in the region 
\[
\Om_s\times(-\dt,\dt):=\{(x_0,p,u_0)\in T^*M\times\R:  |p-p_0|\leq P_s\}\times\{t\in\R:|t|\leq \dt\}.
\]
Moreover, for any $t\in(0,\dt]$, $\pi\circ \Phi_{H,c}^t:\Om_s\rightarrow M$ is a diffeomorphism since $\partial_{pp}H$ is positive definite everywhere. Denote by $\Lb_s^t=\pi\circ\Phi_{H,c}^t(\Om_s)$ for $t\in(0,\dt]$ and 
\[
\Pi_s^\dt:= \{(x,t)\in M\times\R: x\in \Lb_s^t, t\in(0,\dt]\}, 
\]
then for any $(x,t)\in \Pi_s^\dt$, there exists a unique $p_{x,t}\in T^*_{x_0}M$ with $|p_{x,t}-p_0|\leq P_s$, such that  $\pi\circ\Phi_{H,c}^t(x_0,p_{x,t},u_0)=x$. If we denote by $U(x,t):=\pi_u\circ\Phi_{H,c}^t(x_0,p_{x,t},u_0)$\footnote{Here $\pi_u:(x,p,u)\in T^*M\times\R\rightarrow u\in\R$ is the canonical projection to the $u-$component.}, then it's a $C^1-$smooth solution of the following evolutionary equation 
\[
\partial_t\om+H(x,\partial_x\om, \om)=c,\quad {\rm for}\;(x,t)\in \Pi_s^\dt,
\]
see Corollary 2.8.12 of \cite{F}, or \cite{A,WWY} for the proof.
Due to a similar argument as in Proposition \ref{prop:dp},  there holds 
\be\label{eq:arnold}
U(x,t)=\inf_{\substack{(y,\tau)\in \Pi_s^\dt\\\eta\in C^{\rm ac}([\tau, t], \pi(\Pi_s^\dt))\\\eta(\tau)=y,\eta(t)=x}}\Big(U(y,\tau)+\int_\tau^t\big[L\big(\eta(\sigma),\dot\eta, U(\eta(\sigma), \sigma)\big)+c\big] d\sigma\Big).
\ee
Notice that $\lim_{t\rightarrow 0_+}U(x,t)=u(x_0)=u(\gamma(s))$ for any $(x,t)\in\Pi_s^\dt$. So we can define 
\[
F(\tau):=U(\gamma(\tau+s),\tau)-u(\gamma(\tau+s)), 
\]
then $\lim_{\tau\rightarrow 0_+}F(\tau)=0$. We aim to show that $F(\tau)\equiv 0$ for any $\tau \in [0,\min\{-s,\dt\})$. We first prove that $F(\tau)\leq 0$. If not, then
$F(\tau')>0$ for some point $\tau'\in [0,\min\{-s,\dt\})$. Due to the continuity of $F(\cdot)$ there has to be a smallest $\tau_0\in[0,\tau')$, such that $F(\tau_0)=0$ and $F(\tau)>0$ for all $\tau\in(\tau_0,\tau']$. However, for any $t\in(\tau_0,\tau')$, due to Proposition \ref{prop:dp} we have 
\ben
0<F(t)&=&[U(\gamma(t+s),t)-u(\gamma(\tau_0+s))]-[u(\gamma(t+s))-u(\gamma(\tau_0+s))]\\
&=&[U(\gamma(t+s),t)-u(\gamma(\tau_0+s))]-\int_{\tau_0+s}^{t+s}L\big(\gamma(\tau),\dot\gamma(\tau), u(\gamma(\tau))\big) d\tau\\
&=&[U(\gamma(t+s),t)-U(\gamma(\tau_0+s),\tau_0)]-\int_{\tau_0+s}^{t+s}L\big(\gamma(\tau),\dot\gamma(\tau), u(\gamma(\tau))\big) d\tau\\
&\leq &\int_{\tau_0+s}^{t+s}\big[L\big(\gamma(\tau),\dot\gamma(\tau), U(\gamma(\tau), \tau-s)\big)-L\big(\gamma(\tau),\dot\gamma(\tau), u(\gamma(\tau))\big)\big] d\tau\\
&\leq & 0
\een
which is a contradiction. Next, we prove that $F(\tau)\geq 0$. If not, then
$F(\tau')<0$ for some point $\tau'\in [0,\min\{-s,\dt\})$. Due to the definition of $U(x,t)$, there exists a $C^1-$minimizer $\eta: [s,\tau'+s]\rightarrow \pi\Pi_s^\dt$ of \eqref{eq:arnold} connecting $\gamma(s)$ to $\gamma(\tau'+s)$. We define $G(\tau):=U(\eta(\tau+s), \tau)-u(\eta(\tau+s))$, then $G(0)=F(0)=0$ and $G(\tau')=F(\tau')<0$. Due to the continuity of $G(\cdot)$ there exists a smallest $\tau_0 \geq 0$ such that $G(\tau_0)=0$ and $G(\tau)<0$ for all $\tau\in(\tau_0,\tau']$. We can repeat above inequality, i.e. 
\ben
0<-F(\tau')&=&-G(\tau')\\
&=&[u(\eta(\tau'+s))-u(\eta(\tau_0+s))]-[U(\eta(\tau'+s),\tau')-U(\eta(\tau_0+s),\tau_0)]\\
&=&[u(\eta(\tau'+s))-u(\eta(\tau_0+s))]-\int_{\tau_0+s}^{t+s}L\big(\eta(\tau),\dot\eta(\tau), U(\eta(\tau),\tau-s)\big) d\tau\\
&\leq&\int_{\tau_0+s}^{t+s}L\big(\eta(\tau),\dot\eta(\tau), u(\eta(\tau))\big) d\tau-\int_{\tau_0+s}^{t+s}L\big(\eta(\tau),\dot\eta(\tau), U(\eta(\tau),\tau-s)\big) d\tau\\
&\leq & 0
\een
which is also a contradiction.

Recall that $\gamma:(-\infty,0]\rightarrow M$ is a backward $u-$calibrated curve which is differentiable at $s\in(-\infty,0)$, previous argument actually implies 
\[
\frac d{dt} u\circ\gamma(t)=\langle d_xu(\gamma(t)),\dot\gamma(t)\rangle= L\big(\gamma(t),\dot\gamma(t), u(\gamma(t)\big)+c
\]
for all $t\in(s,\min\{0,s+\dt\})$. Due to {\it Young's inequality}, there must be 
\[
\Phi_{H,c}^{t-s}(x_0,p_0,u_0)=\big(\gamma(t), \partial_vL(\gamma(t),\dot\gamma(t), u(\gamma(t))), u(\gamma(t))\big)
\]
for all $t\in(s,\min\{0,s+\dt\})$. Since we can choose $s\in(-\infty,0)$ freely and almost everywhere, then $\big(\gamma(t), \partial_vL(\gamma(t),\dot\gamma(t), u(\gamma(t))), u(\gamma(t))\big)$ is a trajectory of \eqref{eq:ode} (associated with $H_c$) for all $t\in(-\infty,0]$.
\qed\\

\noindent{\it Proof of Theorem \ref{thm:mat-set}:}
Due to Lemma \ref{lem:cali-flow}, any $u-$calibrated curve $\gamma:(-\infty,0]\rightarrow M$ is actually $C^1-$smooth and $(\gamma,\dot\gamma, u(\gamma))$ solves \eqref{eq:e-l}
for all $t\in(-\infty,0)$. For any $\mu\in\mathfrak M(u)$, we can always find a $u-$calibrated curve $\gamma:(-\infty,0]\rightarrow M$ and two sequences $a_n\leq b_n\leq 0$ with $\lim_{n\rightarrow+\infty} b_n-a_n=+\infty$, such that 
\ben
& &\int_{TM\times\R} f(x,v,u) d\mu(x,v,u)\\
&:=&\lim_{n\rightarrow+\infty}\frac{1}{b_n-a_n}\int_{a_n}^{b_n} f(\gamma(t),\dot\gamma(t),u\circ\gamma(t)) dt,\quad\forall\, f\in C_c(TM\times\R,\R).
\een
Since $\overline{\{(\gamma(t), \dot\gamma(t), u\circ\gamma(t))|t\in(-\infty,0)\}}\subset TM\times\R$ is pre-compact, then for any $s\in\R$, 
\ben
\int_{TM\times\R} f(x,v,u) d(\Phi_{L,c}^s)^*\mu(x,v,u)&=&\int_{TM\times\R} f\big(\Phi_{L,c}^s(x,v,u)\big) d\mu(x,v,u)\\
&=&\lim_{n\rightarrow+\infty}\frac{1}{b_n-a_n}\int_{a_n}^{b_n} f(\gamma(t+s),\dot\gamma(t+s),u\circ\gamma(t+s)) dt\\
&=&\lim_{n\rightarrow+\infty}\frac{1}{b_n-a_n}\int_{a_n+s}^{b_n+s} f(\gamma(t),\dot\gamma(t),u\circ\gamma(t)) dt\\
&=&\lim_{n\rightarrow+\infty}\frac{1}{b_n-a_n}\int_{a_n}^{b_n} f(\gamma(t),\dot\gamma(t),u\circ\gamma(t)) dt\\
& &+\lim_{n\rightarrow+\infty}\frac{1}{b_n-a_n}\int_{b_n}^{b_n+s} f(\gamma(t),\dot\gamma(t),u\circ\gamma(t)) dt\\
& &-\lim_{n\rightarrow+\infty}\frac{1}{b_n-a_n}\int_{a_n}^{a_n+s} f(\gamma(t),\dot\gamma(t),u\circ\gamma(t)) dt\\
&=&\lim_{n\rightarrow+\infty}\frac{1}{b_n-a_n}\int_{a_n}^{b_n} f(\gamma(t),\dot\gamma(t),u\circ\gamma(t)) dt\\
&=&\int_{TM\times\R} f(x,v,u) d\mu(x,v,u),
\een
where $\Phi_{L,c}^s:TM\times\R\rightarrow TM\times\R$ is the flow of equation  \eqref{eq:e-l}, which is actually the conjugated equation of \eqref{eq:ode} (associated with $H(x,p,u)-c$) due to the Legendre transformation. Consequently, $\mu$ is invariant w.r.t. $\Phi_{L,c}^t$.

On the other side, due to Corollary 10.3 and Theorem 7.8 of \cite{FS}, the solution $u(x)$ is actually differentiable on $\cM$. Consequently, for any $(x_0,p_0,u_0)\in \wt\cM(u)$, 
\[
\frac d{dt}u(x(t))= \langle d_x u(x(t)),\dot x(t)\rangle,\quad\forall\, t\in\R
\]
along the trajectory $(x(t), p(t), u(t)):=\Phi_{H,c}^t(x_0,p_0,u_0)$ of the equation \eqref{eq:ode}. Since for any initial point, the trajectory of \eqref{eq:ode} (associated with $H(x,p,u)-c$) is unique, then 
\[
d_xu(x(t))=\partial_vL(x(t),\dot x(t), u(t)),\quad \forall t\in\R
\]
due to Young's inequality. That implies $\wt\cM(u)$ can be expressed as in \eqref{eq:graphic}, then due to Theorem \ref{thm:mat-graph}, the map $\pi^{-1}:\cM(u)\rightarrow \wt\cM(u)$ is Lipschitz.
\qed\\

Since any $\mu\in \mathfrak M(u)$  is invariant w.r.t. \eqref{eq:ode} (associated with $H(x,p,u)-c$), then due to the {\it ergodic decomposition theorem} (e.g. see \cite{Pol}), the measure  $\mu$ is actually a convex combination of a family of ergodic Mather measures, i.e. $\mu=\sum_{i\in\Lb}\lb_i\mu_i$ with $\sum_{i\in\Lb}\lb_i=1$ for certain index set $\Lb$. For each ergodic $\mu_i$, there exists a
bilateral $u-$calibrated curve $\gamma_i:\R\rightarrow M$, such that 
\beq
\overline{\{(\gamma_i(t),\dot\gamma_i(t))\in TM|  t\in(-\infty,+\infty) \}}={\rm supp}(\mu_i).
\eeq
The associated $\wt\cM(u)|_{(x,v)\in {\rm supp}(\mu_i)}$ is a topologically minimal set.


%
%

\section{Micro structure of $\mathfrak C$ and its constraint on Mather measures}\label{s5}

Throughout this section, the Hamiltonian $H$ is assumed to satisfy {\bf H1-H4}. \medskip


\noindent{\it Proof of Theorem \ref{thm:h4}:}
1). Recall that $c:\R\rightarrow\R$ is non-decreasing. For any $a<b$,  there holds 
\ben
0\leq c(b)-c(a)&=& \inf_{\mu\in\cC}\int_{TM}L(x,v,a)d\mu-\inf_{\mu\in\cC}\int_{TM}L(x,v,b)d\mu\\
&=&\inf_{\mu\in\cC}\int_{TM}L(x,v,a)d\mu-\int_{TM}L(x,v,b)d\mu_b\\
&\leq&\int_{TM}\big(L(x,v,a)-L(x,v,b)\big)d\mu_b
\een
for any $\mu_b\in\mathfrak M^b$. Due to \cite{CFZZ}, the Mather set $\wt\cM^b\subset TM\times\R$ is uniformly compact for any $b\in[a, a+1]$, which implies 
\[
0\leq c(b)-c(a)\leq \max_{\substack{(x,v)\in\wt\cM^\theta\\\theta\in[a,b]}}|\partial_uL(x,v,\theta)|(b-a).
\]
So $c:\R\rightarrow\R$ is locally Lipschitz and differentiable for almost every point in $\R$.

2). As previously given, for $a<b$
\ben
c(b)-c(a)
&\leq&\int_{TM}\big(L(x,v,a)-L(x,v,b)\big)d\mu_b
\een
for any $\mu_b\in\mathfrak M^b$. On the other side, 
\ben
c(b)-c(a)&=&\inf_{\mu\in\cC}\int_{TM}L(x,v,a)d\mu-\inf_{\mu\in\cC}\int_{TM}L(x,v,b)d\mu\\
&=&\int_{TM}L(x,v,a)d\mu_a-\inf_{\mu\in\cC}\int_{TM}L(x,v,b)d\mu\\
&\geq&\int_{TM}\big(L(x,v,a)-L(x,v,b)\big)d\mu_a
\een
for any $\mu_a\in\mathfrak M^a$. Combining these two inequalities we get 
\beq\label{eq:+}
\int_{TM}\frac{L(x,v,a)-L(x,v,b)}{b-a}d\mu_a\leq \frac{c(b)-c(a)}{b-a}\leq \int_{TM}\frac{L(x,v,a)-L(x,v,b)}{b-a}d\mu_b
\eeq
for any $\mu_a\in\mathfrak M^a$ and $\mu_b\in\mathfrak M^b$. 
Similarly, for $b<a$ we get 
\beq\label{eq:-}
\int_{TM}\frac{L(x,v,a)-L(x,v,b)}{b-a}d\mu_b\leq \frac{c(b)-c(a)}{b-a}\leq \int_{TM}\frac{L(x,v,a)-L(x,v,b)}{b-a}d\mu_a
\eeq
for any $\mu_a\in\mathfrak M^a$ and $\mu_b\in\mathfrak M^b$. 


As a consequence of \eqref{eq:-}, $\mathfrak M_-^\theta\neq\emptyset$ urges $\varlimsup_{\theta'\rightarrow\theta_-}\frac{c(\theta')-c(\theta)}{\theta'-\theta}\leq 0$, whereas due to {\bf H3} $\varliminf_{\theta'\rightarrow\theta_-}\frac{c(\theta')-c(\theta)}{\theta'-\theta}\geq 0$, so $c'_-(\theta)$ exists and equals $0$.

3). Due to \eqref{eq:+} and \eqref{eq:-}, $\mathfrak M_-^\theta=\emptyset$ urges that 
$\varlimsup_{\theta'\rightarrow\theta}\mathfrak M^{\theta'}\cap\mathfrak M_-^\theta=\emptyset$, so item 3) holds. 

4). If $c'(\theta)$ exists, then $c'(\theta)=c'_-(\theta)=c'_+(\theta)$. By taking $\theta'\rightarrow \theta$ from one side in \eqref{eq:+} and \eqref{eq:-} respectively, then we conclude
\[
c'_+(\theta)=\max_{\mu\in\mathfrak M^\theta}\int_{TM}-\partial_uL(x,v,\theta)d\mu(x,v), \quad c'_-(\theta)=\min_{\mu\in\mathfrak M^\theta}\int_{TM}-\partial_uL(x,v,\theta)d\mu(x,v),
\]
so these two formulas coincide. Due to {\bf H3}, if $c'(\theta)>0$, then $\int_{TM}-\partial_uL(x,v,\theta)d\mu(x,v)>0$ for all $\mu\in\mathfrak M^\theta$, so $\mathfrak M_-^\theta=\emptyset$. Besides, $c'(\theta)=\int_{TM}-\partial_uL(x,v,\theta)d\mu(x,v)$ for all  $\mu\in\mathfrak M^\theta$. If $c'(\theta)=0$, we can following the argument and get $\mathfrak M^\theta\backslash\mathfrak M_-^\theta=\emptyset$.

5)-8). These four items can be drawn directly from  \eqref{eq:+} and \eqref{eq:-}, by taking $b\rightarrow a$ from each side. 
%
%
%
\qed

\begin{prop}[ergodic constant 2]\label{prop:mane2}
If {\bf H1-H4, H5'} is assumed, then $c:\R\rightarrow\R$ is convex. Consequently, $\mathfrak C$ formally matches one in the four: $\{c_0\}, [c_0,+\infty)$, $(c_0,+\infty)$ and $(-\infty,+\infty)$ ($c_0\in\R$ is finite).
\end{prop}
\proof
 Since {\bf H5'} is assumed, it suffices to show that for any $a<b$, there holds 
\[
c(\lb a+(1-\lb)b)\leq \lb c(a)+(1-\lb)c(b),\quad\forall\,\lb\in[0,1].
\] 
Indeed, due to \eqref{eq:e-const}, if there exists a subsolution $\om_a$ (resp. $\om_b$) of 
\[
H(x,d_x\om_a(x), a)\leq c(a)\quad {\rm (resp.\;\;} H(x,d_x\om_b(x), b)\leq c(b){\rm )}, 
\]
then for any $\lb\in[0,1]$
\ben
& &H(x,\lb d_x\om_a(x)+(1-\lb)d_x\om_b(x), \lb a+(1-\lb)b)\\
&\leq& \lb H(x,d_x\om_a(x), a) +(1-\lb)H(x,d_x\om_b(x), b)\\
&\leq&  \lb c(a)+(1-\lb) c(b),\quad \quad {\rm a.e.}\; x\in M.
\een
 Once again by \eqref{eq:e-const}, we conclude that $\lb c(a)+(1-\lb) c(b)\geq c(\lb a+(1-\lb) b)$.

Due to Proposition \ref{prop:mane}, $c(a)$ is continuous, non-decreasing and convex. Therefore, 
\[
\sup_{a\in\R}c(a)=+\infty\quad \text{ or }\quad\sup_{a\in\R}c(a) = \inf_{a\in\R}c(a) = c_0
\]
 for some $c_0 \in \R$. If $\inf\mathfrak C=-\infty$, then $\mathfrak C=(-\infty,+\infty)$. If $c_0:=\inf\mathfrak C$ is finite, then either $ c(a)>c_0$ for all $a\in\R$, or there exists some $a_0\in\R$, such that $c(a)=c_0$ for all $a\leq a_0$. For the former case,  $\mathfrak C$ can be established as $(c_0,+\infty)$. For the latter case, $\mathfrak C$ can be established as $\{c_0\}$ or $ [c_0,+\infty)$.\qed\\

\noindent{\it Proof of Corollary \ref{cor:h5-c}:}
Due to {\bf H5'}, \eqref{eq:+} and \eqref{eq:-} can be reinforced to 
\beq\label{eq:+'}
-\int_{TM}\partial_uL(x,v,a)d\mu_a\leq \frac{c(b)-c(a)}{b-a}\leq -\int_{TM}\partial_uL(x,v,b)d\mu_b,\quad (a<b)
\eeq
for any $\mu_a\in\mathfrak M^a$, $\mu_b\in\mathfrak M^b$ and 
\beq\label{eq:-'}
-\int_{TM}\partial_uL(x,v,b)d\mu_b\leq \frac{c(b)-c(a)}{b-a}\leq -\int_{TM}\partial_uL(x,v,a)d\mu_a,\quad(b<a)
\eeq
for any $\mu_a\in\mathfrak M^a$, $\mu_b\in\mathfrak M^b$. Recall that  the single sided derivative of $c:\R\rightarrow\R$ always exists due to Proposition \ref{prop:mane2}, so 
 if $c_0:=\inf\mathfrak C$ is finite and $\theta\in c^{-1}(c_0)$, then $c'_-(\theta)=0$. Therefore, $\mathfrak M_-^\theta\neq\emptyset$, item 1) holds. If $c(\theta)>c_0$, then $0<c'_-(\theta)\leq c'_+(\theta)$ and $\mathfrak M_-^\theta=\emptyset$. So item 2) holds. 

If $\mathfrak M_-^\theta\neq\emptyset$, then $c'_-(\theta)\leq 0$. Since $c:\R\rightarrow\R$ is non-decreasing, that implies $c'_-(\theta)=0$ and then $c(\theta')=c(\theta)$ for any $\theta'<\theta$. That further leads to $c(\theta)=\inf\mathfrak C$, i.e. item 3) is true.

If $\mathfrak M_-^\theta=\emptyset$, then due to \eqref{eq:+} $c'_+(\theta)>0$. 
If $c(\theta)=\inf\mathfrak C$, then $c'_-(\theta)=0$. That urges 
\[
-\int_{TM}\partial_uL(x,v,a)d\sigma_\theta\leq 0
\]
for any $\sigma_\theta\in \varlimsup_{\theta'\rightarrow \theta_-}\mathfrak M^{\theta'}$. However, $\varlimsup_{\theta'\rightarrow \theta_-}\mathfrak M^{\theta'}\subset\mathfrak M^\theta$ and $\mathfrak M_-^\theta=\emptyset$, so due to the compactness of $\wt\cM^\theta\subset TM\times\R$, there exists a constant $\dt>0$ such that 
 \[
0<\dt\leq -\int_{TM}\partial_uL(x,v,a)d\sigma_\theta\leq 0
\]
for any $\sigma_\theta\in\varlimsup_{\theta'\rightarrow \theta_-}\mathfrak M^{\theta'}$, which is a contradiction. So item 4) holds. 
\qed

\vspace{40pt}


\end{document}